\documentclass[reqno,11pt]{amsart}
\usepackage{cite}
\usepackage{mathrsfs}
\usepackage{amssymb}
\usepackage{amsthm}
\usepackage[OT1]{fontenc}
\usepackage[latin1]{inputenc}
\usepackage{epsf}            
\usepackage{hyperref}
\usepackage[usenames]{color}
\usepackage[all]{xy}
\CompileMatrices
\usepackage{epic}
\usepackage{eepic}
 \usepackage{enumerate}
 \usepackage[active]{srcltx}
  \usepackage{graphicx}
  \usepackage{ae,aecompl}
\usepackage{comment}
\usepackage{appendix}

\numberwithin{equation}{section}

\DeclareMathAlphabet{\realcal}{U}{rsfs}{m}{n}

\def\ci#1{\ensuremath{{\mathcal {#1}}}}

\def\pv#1{\ensuremath{{\sf#1}}}

\def\J{\ensuremath{\mathscr J}}
\def\R{\ensuremath{\mathscr R}}
\renewcommand{\L}{\ensuremath{\mathscr L}}

\newtheorem{Thm}{Theorem}[section]

\newtheorem{Prop}[Thm]{Proposition}
\newtheorem{Lemma}[Thm]{Lemma}
\newtheorem{Cor}[Thm]{Corollary}
{\theoremstyle{remark}

}
{\theoremstyle{remark}
}
{\theoremstyle{definition}
}
{\theoremstyle{definition}
\newtheorem{Def}[Thm]{Definition}}
{\theoremstyle{remark}
\newtheorem{Rmk}[Thm]{Remark}}

\newcommand{\DD}{\ensuremath{\mathbb D}}
\newcommand{\KK}{\ensuremath{\mathbb K}}

\newcommand{\inv}{^{-1}}

\begin{document}

\title{The Sch\"utzenberger category of a semigroup}
\thanks{The first author was supported by the Centro
  de Matem\'atica da Universidade de Coimbra (CMUC),
  funded by the European Regional Development Fund through
  the program COMPETE and by the Portuguese Government
  through the Funda\c {c}\~ao para a Ci\^encia e a
Tecnologia (FCT)
under the project
PEst-C/MAT/UI0324/2013. He was also supported by the FCT post-doctoral grant
  SFRH/BPD/46415/2008 and by the FCT
  project PTDC/MAT/65481/2006, within the framework of European programmes
 COMPETE and FEDER.  The second author was partially supported by a grant from the Simons Foundation (\#245268 to Benjamin Steinberg)
and the Binational Science Foundation of Israel and the US (\#2012080 to Benjamin Steinberg). Some of this work was performed while the second author was at the School of Mathematics and Statistics, Carleton University, Ottawa, Canada under the auspices of an NSERC grant}
\author{Alfredo Costa}
\address{CMUC, Department of Mathematics, University of Coimbra,
  3001-501 Coimbra, Portugal.}
\email{amgc@mat.uc.pt}
\author{Benjamin Steinberg}
\address{Department of Mathematics\\
City College of New York\\
NAC 8/133\\
Convent Ave at 138th Street\\
New York, NY 10031}
\email{bsteinberg@ccny.cuny.edu}
\date{\today}

\begin{abstract}
In this paper we introduce the Sch\"utzenberger category $\DD(S)$ of a semigroup $S$.  It stands in relation to the Karoubi envelope (or Cauchy completion) of $S$ in the same way that Sch\"utzenberger groups do to maximal subgroups and that the local divisors of Diekert do to the local monoids $eSe$ of $S$ with $e\in E(S)$.  In particular, the objects of $\DD(S)$ are the elements of $S$, two objects of $\DD(S)$ are isomorphic if and only if the corresponding semigroup elements are $\mathscr D$-equivalent, the endomorphism monoid at $s$ is the local divisor in the sense of Diekert and the automorphism group at $s$ is the Sch\"utzenberger group of the $\mathscr H$-class of $S$.  This makes transparent many well-known properties of Green's relations.

The paper also establishes a number of technical results about the Karoubi envelope and Sch\"utzenberger category that were used by the authors in a companion paper on syntactic invariants of flow equivalence of symbolic dynamical systems.
\end{abstract}

  \makeatletter
  \@namedef{subjclassname@2010}{%
    \textup{2010} Mathematics Subject Classification}
  \makeatother

\maketitle


\section{Introduction}
Small categories have played an important role in semigroup theory for a long time, see for example~\cite{Tilson,Cats2,loganathan,Leech:1975}.  One of the first categories considered in this context was the Karoubi envelope $\KK(S)$ of a semigroup $S$; it plays a fundamental role in Tilson's Delay Theorem~\cite{Tilson} (see~\cite{Eilenberg} and~\cite{Straubingdelay} for its predecessors).  Informally, the objects of $\KK(S)$ are idempotents of $S$ and $\KK(S)(e,e')= e'Se$.  Composition is induced by the multiplication in $S$.  It is easy to check that isomorphism in $\KK(S)$ corresponds to $\mathscr D$-equivalence, the endomorphism monoids are the so-called local monoids $eSe$ and the automorphism groups are the maximal subgroups.

Sch\"utzenberger, long ago, associated a group, now called the Sch\"utzenberger group~\cite{CP}, to each $\mathscr H$-class of a semigroup which coincides with the maximal subgroup when considering the $\mathscr H$-class of an idempotent.  Recently, Diekert and his coauthors have associated a monoid to each element $s$ of a semigroup $S$, called a local divisor~\cite{GastinDiekert,newKR,diekertchurch}, which in the case that $s=e$ is an idempotent reduces to $eSe$ and whose group of units is the Sch\"utzenberger group of the $\mathscr H$-class of $s$.

Here we introduce a small category $\DD(S)$, which we call the Sch\"utzenberger category of $S$, whose objects are elements of $S$.  The endomorphism monoid at an object $s$ will be Diekert's local divisor associated to $s$ and the automorphism group will be the Sch\"utzenberger group of the $\mathscr H$-class of $s$.  Two objects of $\DD(S)$ are isomorphic if and only if they are $\mathscr D$-equivalent (whence the notation).  The Karoubi envelope is isomorphic to the full subcategory of $\DD(S)$ on the idempotents.  The construction satisfies $\DD(S^{op})=\DD(S)^{op}$ and so we obtain a transparent proof that the left and right Sch\"utzenberger groups of an $\mathscr H$-class are isomorphic and depend only on the $\mathscr D$-class up to isomorphism.

This paper is intended as a companion to~\cite{categoricalinvariant}.
In that paper, it is shown that if $\mathscr X$ is a symbolic
dynamical system and $S(\mathscr X)$ is the syntactic semigroup of the
language of blocks of $\mathscr X$, then $\KK(S(\mathscr X))$ (up to
natural equivalence) is a flow equivalence invariant of $\mathscr X$.
Flow equivalence is a highly studied coarsening of topological
conjugacy.  See~\cite{MarcusandLind} for details on symbolic dynamics.
In~\cite{categoricalinvariant}, we needed some purely algebraic
properties of $\KK(S)$ and $\DD(S)$ whose proofs were relegated to this paper, as they had nothing to do with dynamics.

The paper is organized as follows. After a brief section on categorical preliminaries, we introduce the Karoubi envelope and Sch\"utzenberger category of a semigroup, first as a concrete category, and then abstractly. The following section discusses the connection between actions of semigroups and actions of their associated categories.  We also prove the technical result that if $S$ and $T$ have local units, then $\KK(S)$ equivalent to $\KK(T)$ implies that $\DD(S)$ is equivalent to $\DD(T)$.  In the final section, we consider several invariants of an action of a semigroup on a set that are detected by the Karoubi envelope.  These invariants are used in~\cite{categoricalinvariant} to obtain flow equivalence invariants of symbolic dynamical systems.


\section{Categorical preliminaries}
The reader is referred to~\cite{Mac-CWM,Borceux1,Borceux2,Borceux3} for basic notions from
category theory.  A category $C$ is \emph{small} if its objects and
arrows form a set. We often write $C(c,c')$ for the set of morphisms from $c$ to $c'$.  Note that we compose arrows as traditionally done by category theorists: if $f\colon c\to c'$ and $g\colon c'\to c''$, then their composition is $gf\colon c\to c''$. The identity at an object $c$ is usually denoted $1_c$.

 A natural transformation $\eta\colon F\Rightarrow G$ between functors $F,G\colon C\to D$ is an assignment of an arrow $\eta_c\colon F(c)\to G(c)$ for each object $c$ of $C$ such that, for any arrow $f\colon c\to c'$ of $C$, the diagram
\[\xymatrix{F(c)\ar[r]^{F(f)}\ar[d]_{\eta_c} & F(c')\ar[d]^{\eta_{c'}}\\ G(c)\ar[r]_{G(f)} & G(c')}\]
commutes.  There is then a category $D^C$ with objects  functors from $C$ to $D$ and arrows natural transformations between them. Two functors $F,G\colon C\to D$ are isomorphic, written $F\cong G$, if they are isomorphic in $D^C$.

We recall that two categories $C$ and $D$ are
\emph{equivalent} if there are functors $F\colon C\to D$ and $G\colon
D\to C$ such that $FG\cong 1_D$ and $GF\cong 1_C$.
A functor $F\colon
C\to D$ between small categories is an equivalence (i.e., there exists
such a $G$, called a \emph{quasi-inverse}) if and only if it is fully faithful and essentially
surjective.  \emph{Fully faithful} means bijective on hom-sets (i.e., $F\colon C(c,c')\to D(F(c),F(c'))$ is a bijection for all $c,c'$),
whereas \emph{essentially surjective} means that every object of $D$
is isomorphic to an object in the image of $F$.
The former is in accordance with the usual terminology
for functors which are injective on hom-sets (the \emph{faithful} functors)
and for those which are surjective on hom-sets (the \emph{full} functors).

\section{The Karoubi envelope and Sch\"utzenberger category of a semigroup}


\subsection{Inner morphisms of principal right ideals}
Let $S$ be a semigroup.  A right \emph{$S$-set} is a set $X$ equipped with a right action by $S$. Key examples include principal right ideals $sS^1$ (where $S^1$ is the result of adjoining an external identity to $S$). Note that if $s$ is regular, then $sS^1=sS$ and we often write the latter. If $S$ is a monoid (or more generally has local units, see below) then the projective indecomposable $S$-sets are precisely those isomorphic to one of the form $eS^1=eS$ for some idempotent $e\in S$.


A \emph{morphism} of right $S$-sets is a mapping $f\colon X\to Y$ such
that $f(xs)=f(x)s$ for all $x\in X$ and $s\in S$. We say that a
morphism $f\colon sS^1\to tS^1$ of principal right ideals
is \emph{inner} if there exists $u\in S^1$ such that $f(x)=ux$ for all
$x\in sS^1$.  We do not claim that $u$ is unique. By an \emph{inner isomorphism}, we mean one whose inverse is also inner. Let us say that $sS^1$ and $tS^1$ are \emph{inner isomorphic} if there is an inner isomorphism between them. We write $\hom(sS^1,tS^1)$ for the set of inner morphisms $sS^1\to tS^1$.

We record here some basic facts about morphisms.

\begin{Prop}\label{p:innerbasic}
Let $S$ be a semigroup.
\begin{enumerate}
\item If $f,g\colon sS^1\to tS^1$ are morphisms, then $f=g$ if and only if $f(s)=g(s)$.
\item Left multiplication by an element $u\in S^1$ induces an inner morphism $sS^1\to tS^1$ if and only if $u\in (tS^1)s^{-1}=\{v\in S^1\mid vs\in tS^1\}$, if and only if $us\in tS^1$.
\item Two elements $u,v$ of $(tS^1)s^{-1}$ induce the same inner morphism if and only if $us=vs$.
\item The mapping $f\mapsto f(s)$ gives a bijection $\hom(sS^1,tS^1)\to tS^1\cap S^1s$.
\item If $e,f$ are idempotents, then $\hom(eS,fS)$ is in bijection with  $fSe=fS\cap Se$ via $g\mapsto g(e)$.
\item If $s,t$ are regular, then every morphism $sS\to tS$ is inner.
\item The composition of inner morphisms is inner and hence we can
  consider the category of principal right ideals of $S$ with inner
  morphisms. More precisely, if $v$ induces an inner morphism $sS^1\to
  tS^1$ and $w$ induces an inner morphism $tS^1\to uS^1$, then their
  composition $sS^1\to uS^1$ is the inner morphism induced by $wv$.
\item An inner morphism $f\colon sS^1\to tS^1$ corresponding to $w=f(s)\in
  tS^1\cap S^1s$ is an inner isomorphism if and only if
  $t\mathrel{\mathscr R} w\mathrel{\mathscr L} s$.
\item $s\mathrel{\mathscr D} t$ if and only if $sS^1$ and $tS^1$ are inner isomorphic.
\item The action of the group of inner automorphisms of $sS^1$ restricts to a free and transitive action on the $\mathscr H$-class $H_s$.  Thus we
  can identify the group of inner automorphisms of $sS^1$ with the
  left Sch\"utzenberger group
  of~$H_s$.
\end{enumerate}
\end{Prop}
\begin{proof}
If $f(s)=g(s)$, then $f(sx)=f(s)x=g(s)x=g(sx)$ for all $x\in S^1$ establishing the first item. If $x\mapsto ux$ is an inner morphism $sS^1\to tS^1$, then in particular $us\in tS^1$. Conversely, if $us\in tS^1$, then obviously $usx\in tS^1$ for all $x\in S^1$.  This establishes the second item.  The third item is immediate from the first. If $f\in \hom(sS^1,tS^1)$ with $f(x)=ux$, then $f(s)=us\in tS^1\cap S^1s$ and the map $f\mapsto f(s)$ is injective by the first item. If $v\in tS^1\cap S^1s$ and $v=us$ with $u\in S^1$, then by the second item, $u$ induces an inner morphism $f\colon sS^1\to tS^1$ with $f(s)=us=v$.  This proves the fourth item.
The fifth item is a special case of the fourth item.

Suppose that $ss's=s$ and $tt't=t$.  Let $f\colon sS\to tS$ be a morphism and put $u=f(ss')$. Then $us=f(ss')s=f(ss's)=f(s)\in tS^1$. Thus $u$ induces an inner morphism which agrees with $f$ on $s$ and hence on $sS^1$ by the first item. This proves the sixth item.  The seventh item is trivial.

For the eighth item, suppose first that $f$ has inner inverse $g$.  Say that $f(x)=ux$ and $g(y)=vy$.  Then $f(s)=us$ and $s=gf(s)=vf(s)$.  Thus $s\mathrel{\mathscr L} f(s)$.  Clearly, $f(s)\in tS^1$.  Since $f$ is surjective, there exists $z\in S^1$ with $t=f(sz)=f(s)z$.  Thus $f(s)\mathrel{\mathscr R} t$.
Conversely, suppose that $f$ is an inner morphism with $t\mathrel{\mathscr R}f(s)\mathrel{\mathscr L} s$. Say $f(x)=ux$. Then $f(s)=us$ and so we have $s=yus$ for some $y\in S^1$.  Then since $yus\in sS^1$, left multiplication by $y$ induces an inner morphism $g\colon tS^1=usS^1\to sS^1$.  We claim that $f,g$ are inverses.  Indeed, $gf(s)=g(us)=yus=s$ and so $gf$ is the identity by the first item.  Similarly, $fg(us)=f(yus)=f(s)=us$.  We again conclude from the first item that $fg$ is the identity.

To prove the ninth item, we directly observe from the eight item that if $sS^1$ and $tS^1$ are inner isomorphic, then $s\mathrel{\mathscr D}t$.  Conversely, if $s\mathrel{\mathscr L} y\mathrel{\mathscr R}t$, then $y=us$ for some $u\in S^1$.  Thus $f\colon sS^1\to tS^1$ given by $f(x)=ux$ is an inner morphism with $f(s)=us=y$ and $s\mathrel{\mathscr L} y\mathrel{\mathscr R}t$.  Thus $f$ is an inner isomorphism by the previous item.

For the final item, let $G$ be the group of inner automorphisms of
$sS^1$.  If $g\in G$, then
 $s\mathrel{\mathscr H} g(s)$ by item (8).
As the $\mathscr H$-class $H_s$ of $s$ is contained in the set of generators of $sS^1$ as a right
ideal, it follows that $H_s$ is $G$-invariant.
The action on $H_s$ is free by (1).   If $w\in H_s$, then  the inner morphism $f\colon sS^1\to sS^1$ with $f(s)=w$ is an inner automorphism by (8) and so $G$ is transitive on $H_s$.
Recall that the left Sch\"utzenberger group $G(H_s)$ consists of all
mappings $f\colon H_s\to H_s$ such that, for some $u\in S^1$, one has
$f(x)=ux$ for all $x\in H_s$.  But then $us\in H_s$ implies that $f$ extends uniquely to an inner automorphism of $sS^1$ by (1) and (8).  We conclude that $f\mapsto f|_{H_s}$ gives an isomorphism of $G$ with $G(H_s)$.
\end{proof}

Dual notions and results hold for principal left ideals.

Proposition~\ref{p:innerbasic} provides a new proof that $s\mathrel{\mathscr D} t$ implies that the left Sch\"utzenberger groups of $H_s$ and $H_t$ are isomorphic.  The classical proof makes use of right Sch\"utzenberger groups.  We will soon have a new proof that the right and left Sch\"utzenberger groups are isomophic.

\subsection{The Karoubi envelope of a semigroup}
We now present an algebraic model of the category of
idempotent-generated principal right ideals of $S$ and their morphisms
(all of which are inner
by Proposition~\ref{p:innerbasic}).  The construction is quite well known to category theorists in a more general context (cf.~\cite{Borceux1}). In semigroup theory, this category was first considered by Tilson~\cite{Tilson} in connection with the Delay Theorem; see also~\cite[Chapter 4]{qtheor} where some basic properties are exposited.

Let $S$ be a semigroup. The \emph{Karoubi envelope}
$\KK(S)$ (also known as Cauchy completion or idempotent splitting) of a
semigroup $S$ is a small category whose object set is the set
$E(S)$ of idempotents of $S$.  A  morphism from $f$ to $e$ is a triple
$(e,s,f)$ where $s\in eSf$. Such a morphism will be represented by an arrow
$e\longleftarrow f$ for reasons to be explained
in the next paragraph.
Composition of morphisms is given
by \[(e,s,f)(f,t,g)=(e,st,g).\]  The identity at $e$ is $(e,e,e)$

In Tilson's work, and in most subsequent work in Semigroup Theory,
a triple $(e,s,f)$ such that $s\in eSf$ is viewed as
a morphism with domain $e$ and codomain $f$, and the composition of morphisms
is taken from left to right, which is the opposite of the convention usually followed in
Category Theory.
Accordingly to this alternate convention, we still
have $(e,s,f)(f,t,g)=(e,st,g)$.
In this paper, we adopt the Category Theory convention for composition. However, we do not want
to deviate  from the Semigroup Theory notation, and that is why
we represent graphically
a morphism $(e,s,f)$ as an arrow $e\longleftarrow f$ with the source on the
right.
This same convention will also be used later for the Sch\"utzenberger
category of a semigroup.  All other categories will be treated as usual with arrows
drawn from left to right.

If $e$ is an idempotent of the semigroup $S$, then $eSe$ is a
a monoid with identity~$e$, which is called the \emph{local monoid of
  $S$ at $e$}.
The \emph{local monoid of a category $C$ at an object $c$}
is the endomorphism monoid of $c$ in $C$.
The local monoids of $S$ correspond to the local monoids
of $\KK(S)$, more precisely, $eSe$ and $\KK(S)(e,e)$ are isomorphic for
every $e\in E(S)$.

The following result is well known.
\begin{Thm}\label{t:karoubiasprojindec}
The Karoubi envelope $\KK(S)$ of a semigroup $S$ is equivalent to the category of idempotent-generated principal right ideals of $S$ and their $S$-set morphisms via the functor given on objects by $e\mapsto eS$.
\end{Thm}
\begin{proof}
Let $C$ be the latter category and define a functor $F\colon \KK(S)\to
C$ on objects by $e\mapsto eS$ and on arrows by $(e,s,f)\mapsto
(x\mapsto sx)$. Note that $F$ is surjective on objects and sends $\KK(S)(f,e)$ bijectively to $\hom(fS,eS)$ by Proposition~\ref{p:innerbasic}.  It remains to show that $F$ is a functor.  Note that $F(e,e,e)$ is the map $eS\to eS$ sending $x\to ex$, but $ex=x$ for all $x\in eS$ and so $F(e,e,e)=1_{eS}$. Next we compute $F(e,s,f)F(f,t,g)(x) = s(tx)=(st)x=F(e,st,g)(x)$.  This completes the proof that $F$ is an equivalence.
\end{proof}

Note that $\KK(S)^{op}=\KK(S^{op})$ and hence the category of
idempotent-generated principal left ideals is equivalent to the dual of the
category of idempotent-generated principal right ideals.

It follows from Proposition~\ref{p:innerbasic} that two objects $e,f$ of $\KK(S)$ are isomorphic if and only if they are $\mathscr D$-equivalent.

An element $s$ of a semigroup $S$ has \emph{local units $e$ and $f$}, where $e$
and $f$ are idempotents of $S$, if
$s=esf$.
The set ${LU}(S)=E(S)SE(S)$ of elements of $S$ with
local units is a subsemigroup of $S$.
If ${LU}(S)=S$, then
we say that $S$ has \emph{local units}.
In general, ${LU}(S)$ is the largest subsemigroup of $S$ which has
local units (it may be empty).  Clearly $\KK(S)=\KK({LU}(S))$ and so the
Karoubi envelope does not distinguish between $S$ and $LU(S)$.

Talwar defined in~\cite{Talwar3} a notion of Morita equivalence of
semigroups with local units in terms of equivalence of certain
categories of actions.  It was shown
in~\cite{Lawson:2011,FunkLawsonSteinberg} that semigroups
$S$ and $T$ with local units are Morita equivalent if and only if
$\KK(S)$ and $\KK(T)$ are equivalent categories.

\subsection{The Sch\"utzenberger category of a semigroup}\label{sec:greens-relat-relat}
Next, we define an algebraic model of the category of principal right ideals of $S$ with inner morphisms.  We shall term it the Sch\"utzenberger category of $S$ and denote it $\DD(S)$ (because of its connection with Green's $\mathscr D$ relation).  The reason for the name is that the notion of inner morphism generalizes the key notion in the definition of the Sch\"utzenberger group.  Also, the relationship of the Sch\"utzenberger category to the Karoubi envelope can be viewed as an extension of the relationship of Sch\"utzenberger groups to maximal subgroups.

The \emph{Sch\"utzenberger  category} $\DD(S)$ of a semigroup $S$
is defined as follows.
The object set of $\DD(S)$ is $S$.  A morphism from $t$ to $s$ is a
triple $(s,u,t)$ such that $u\in sS^1\cap S^1t$. The composition
of morphisms is given by
$(s,xt,t)(t,ty,r) = (s,xty,r)$ (with $x,y\in S^1$), which the reader readily verifies is well defined. The identity at $s$ is
$(s,s,s)$.

%

\begin{Rmk}\label{r:subcategory}
  The category $\KK(S)$ is the full subcategory of
  $\DD(S)$ whose objects are the idempotents of $S$. Indeed, $(e,xf,f)(f,fy,g)=(e,xfy,g)=(e,xffy,g)$ when $e,f$ are idempotents.
\end{Rmk}

\begin{Rmk}
The local monoids of the Sch\"utzenberger category has been studied by Diekert and his collaborators under the name \emph{local divisors} and put to good effect; see for example~\cite{GastinDiekert,newKR,diekertchurch}. The Sch\"utzenberger category is novel to this paper.
\end{Rmk}


\begin{Thm}\label{t:Schutzasinnermorphisms}
Let $S$ be a semigroup.  Then $\DD(S)$ is equivalent to the category of principal right ideals of $S$ and inner morphisms via the functor given on objects by $s\mapsto sS^1$.
\end{Thm}
\begin{proof}
Let $C$ be the latter category and define a functor $F\colon \DD(S)\to C$ on objects by $s\mapsto sS^1$ and on arrows by $(s,ut,t)\mapsto (x\mapsto ux)$ for $u\in S^1$. First note that this is well defined because if $ut=vt$, then $ux=vx$ for all $x\in tS^1$.  Also it follows from Proposition~\ref{p:innerbasic} that $x\mapsto ux$ is an inner morphism.  The mapping $F$ is surjective on objects and is a bijection on hom sets by Proposition~\ref{p:innerbasic}.  It remains to prove that $F$ preserves identities and composition. We have that $F(s,s,s)=F(s,1s,s)=(x\mapsto x)$ and so $F$ preserves identities.  Consider composable arrows $(s,a,t),(t,b,v)$ and write $a=ut$ and $b=ty=zv$.  Then $F(s,a,t)F(t,b,v)(x) = uzx$.  On the other hand $(s,a,t)(t,b,v)=(s,uty,v)=(s,uzv,v)$. But $F(s,uzv,v)(x)=uzx$ and so $F$ is a functor.  This completes the proof.
\end{proof}

Trivially, $\DD(S)^{op}=\DD(S^{op})$, from which we obtain that the category of principal left ideals of $S$ and inner morphisms is dual to the category of principal right ideals of $S$ and inner morphisms.  It thus follows from Proposition~\ref{p:innerbasic}, and the fact that dualizing a category preserves automorphism groups, that the left and right Sch\"utzenberger groups of an $\mathscr H$-class are isomorphic.

As a consequence of Theorem~\ref{t:Schutzasinnermorphisms} (and its proof) and Proposition~\ref{p:innerbasic}, we obtain the following lemma.

\begin{Lemma}\label{l:isomorphism-criterion-for-C(S)}
  Let $s,t,u\in  S$.
  The triple $(s,u,t)$ is an isomorphism of
  $\DD(S)$ if and only if
  $s \mathrel{\R} u \mathrel{\L} t$.
  In particular, $s$ and $t$ are
  isomorphic objects of $\DD(S)$
  if and only if they are
  $\mathscr D$-equivalent elements of $S$.
\end{Lemma}



We shall use Lemma~\ref{l:isomorphism-criterion-for-C(S)}
frequently, often without explicit reference.  A consequence of this lemma is that the inclusion $\KK(S)\to \DD(S)$ is an equivalence if and only if $S$ is regular.  Another consequence of Theorem~\ref{t:Schutzasinnermorphisms} and Proposition~\ref{p:innerbasic} is the following lemma.

\begin{Lemma}\label{l:automorphism-group}
  Let $s\in S$. The Sch\"utzenberger group of $s$ in $S$ is isomorphic
  to the automorphism group of $s$ in the category $\DD(S)$.
\end{Lemma}

We mention that the Sch\"utzenberger category has some properties in common with
another category built from $S$, introduced in~\cite{Leech:1975}.
Despite their similarities, there are important differences
between these two categories. For example, when $S$ is stable (e.g., if $S$ is finite or compact; see~\cite[Appendix]{qtheor})
all endomorphisms in the category defined
in~\cite{Leech:1975} are automorphisms.

\section{Lifting a functor between Karoubi envelopes to Sch\"utzenberger categories}\label{sec:lift-funct-fcol}

Given a category $C$ and a semigroup $T$, a \emph{semi-functor}
$F\colon C\to T$ is a mapping from the arrows of $C$ to $T$ satisfying
$F(fg)=F(f)F(g)$ whenever $f,g$ are composable arrows.
There is no requirement on what $F$ does to identities.

Let $\gamma_D\colon \DD(T)\to T$ be the map
sending a morphism $(s,u,t)$ to $u$.  We do not claim that $\gamma_D$ is a semi-functor, however
the restriction $\gamma_K$ of $\gamma_D$
to $\KK(T)$ is a faithful semi-functor (that is, injective on hom-sets).
If $F$ is a functor from a category $C$ into
$\DD(T)$,
we denote by $F_m$
the composition $\gamma_D\circ F$. We also use
the notation $F_m$ for a composition of the form $\gamma_K\circ F$.
The subscript $m$ stands
for \emph{middle}, because if
$\alpha$ is a morphism of $C$ then~$F_m(\alpha)$
is the second component of the triple~$F(\alpha)$.

The definition of the $\J$-order on a semigroup can easily be extended
to categories: if $f,g$ are arrows of a category $C$, then $f\leq_\J
g$ if there are arrows $a,b$ with $f=agb$.  Similar remarks apply to
the $\L$- and $\R$-orders and the $\mathscr D$-relation.

\begin{Prop}\label{reflectsJ}
Let $F\colon C\to D$ be an equivalence of categories.  Then $F$ reflects the $\J$-order, that is,   $F(f)\leq_\J F(g)$ if and only if $f\leq_\J g$.
\end{Prop}
\begin{proof}
If $f=agb$, then $F(f)=F(a)F(g)F(b)$.  This establishes the first statement.  Suppose that $G$ is a quasi-inverse of $F$ and let $\eta\colon 1_C\Rightarrow GF$ be a natural isomorphism.  If $F(f)\leq_\J F(g)$, then we can find arrows $a,b$ of $D$ with $aF(g)b=F(f)$.  Thus $G(a)GF(g)G(b)=GF(f)$, that is, $GF(f)\leq_\J GF(g)$.  It thus suffices to observe that, for any arrow $x\colon c\to d$ of $C$, one has $x\mathrel{\J} GF(x)$.  But commutativity of
\[\xymatrix{c\ar[r]^x\ar[d]_{\eta_c} & d\ar[d]^{\eta_d}\\ GF(c)\ar[r]_{GF(x)} & GF(d)}\] implies that $x=\eta_d^{-1}GF(x)\eta_c$ and $GF(x) = \eta_dx\eta_c^{-1}$, thereby completing the proof.
\end{proof}

Our next proposition shows how the $\J$-order of $S$ is encoded in $\DD(S)$ if $S$ is a semigroup with local units.

\begin{Prop}\label{Jorderinlocalunitscase}
Let $S$ be a semigroup with local units and $s,t\in S$.  Then the following are equivalent:
\begin{enumerate}
\item $s\leq_\J t$;
\item $(a,s,b)\leq_\J (c,t,d)$ in $\DD(S)$ for some $a,b,c,d\in S$;
\item $(a,s,b)\leq_\J (c,t,d)$ in $\DD(S)$ for all $a,b,c,d\in S$ with $as=s=sb$ and $ct=t=td$.
\end{enumerate}
\end{Prop}
\begin{proof}
Trivially (3) implies (2) since $S$ has local units.  To see that (2) implies (1), let $(a,s,b) = (a,u,c)(c,t,d)(d,v,b)$.  Then we can write $u=xc$ and $v=dy$ with $x,y\in S^1$.  Then $s=xty$ and so $s\leq_\J t$.  Suppose that (1) holds, say $s=xty$ with $x,y\in S^1$.  Let $a,b,c,d$ be as in (3).  Then $(a,axc,c)(c,t,d)(d,dyb,b) = (a,axtyb,b)$.  But $axtyb=asb=s$, so (3) holds.
\end{proof}


\begin{Def}
Let $S$ and $T$ be semigroups with local units. We say that
a functor $F\colon \DD(S)\to \DD(T)$ is \emph{good} if
\begin{enumerate}
\item $F$ restricts to an equivalence $\KK(S)\to \KK(T)$;
\item for all $s\in S$, there exist idempotents $e_s,f_s$ with $e_ssf_s=s$ and
\begin{align*}
F(e_s,s,s) &=(F(e_s),F(s),F(s)),\\ F(s,s,f_s)&=(F(s),F(s),F(f_s)).
\end{align*}
\end{enumerate}
\end{Def}

Let us say that a functor $F\colon \DD(S)\to \DD(T)$ \emph{reflects regularity} if $F(s)$ is regular if and only if $s$ is regular for $s\in S$.  We say that it \emph{reflects the $\J$-order on objects} if $s\leq_\J s'$ if and only only if $F(s)\leq_\J F(s')$ for all $s,s'\in S$.

\begin{Lemma}\label{goodequivalence}
Let $S$ and $T$ be semigroups with local units and let $F\colon
\DD(S)\to \DD(T)$ be a good functor.
Then $F$ is an equivalence reflecting regularity and the
$\J$-order on objects.
\end{Lemma}
\begin{proof}
Fix, for each $s\in S$, idempotents $e_s,f_s$ as in the definition of a good functor.
First we show that $F$ is an equivalence.  To prove that $F$ is faithful suppose that $F(a,s,b)=F(a,s',b)$.  Choose idempotents $e,f\in E(S)$ with $ea=a$ and $bf=b$.  Then since $(e,s,f)=(e,a,a)(a,s,b)(b,b,f)$ and $(e,s',f) = (e,a,a)(a,s',b)(b,b,f)$ it follows that $F(e,s,f)=F(e,s',f)$ and so $s=s'$ since $F|_{\KK(S)}$ is faithful.  Thus $F$ is faithful.

To see that $F$ is full, let $(F(a),t,F(b))$ be a morphism of $\DD(T)$. From $(e_s,s,f_s)=(e_s,s,s)(s,s,f_s)$, we obtain that
\begin{equation}\label{olddefgood}
\begin{split}
F(e_s,s,f_s) &= (F(e_s),F(s),F(s))(F(s),F(s),F(f_s))\\ &= (F(e_s),F(s),F(f_s))
\end{split}
\end{equation}
for all $s\in S$.
  By definition of $\DD(T)$, we have
$t=F(a)x=yF(b)$ for some $x,y\in T^1$.
Because $F(e_c,c,f_c)$ is a morphism of $\KK(T)$ for $c=a,b$,
we have $t=F(a)F(f_a)xF(f_b)$, and so without loss of generality we may assume that $x=  F(f_a)xF(f_b)$.
Since $F|_{\KK(S)}$ is full, there is $x'$
such that $F(f_a,x',f_b)=(F(f_a),x,F(f_b))$.  Similarly, we
can assume $y=F(e_a)yF(e_b)$ and find $y'$ such that
$F(e_a,y',e_b)=(F(e_a),y,F(e_b))$.  Now
 \begin{align*}
 F(e_a,ax',f_b) &= F(e_a,a,f_a)F(f_a,x',f_b)\\ &= (F(e_a),F(a),F(f_a))(F(f_a),x,F(f_b)) = (F(e_a),t,F(f_b)),
 \end{align*}
 and similarly $F(e_a,y'b,f_b)=(F(e_a),t,F(f_b))$.  Thus, since $F|_{\KK(S)}$ is faithful, we conclude $ax'=y'b$.
Therefore, $(a,ax',b)$ is a morphism of $\DD(S)$. Let $F(a,ax',b)=(F(a),v,F(b))$.  Now $(e_a,a,a)(a,ax',b)(b,b,f_b) = (e_a,ax',f_b)$  and the definition of a good functor imply that
\begin{align*}
(F(e_a),t,F(f_b))&= F(e_a,ax',f_b)\\ &=(F(e_a),F(a),F(a))F(a,ax',b)(F(b),F(b),F(e_b))\\ &= (F(e_a),v,F(f_b)).
\end{align*}
Thus $v=t$ and $(F(a),t,F(b))=F(a,ax',b)$, and so we conclude that $F$ is full.

Finally, we prove that $F$ is essentially surjective.
Let $t$ be an element of $T$.
Then $t=e_0tf_0$ for some idempotents $e_0,f_0\in E(T)$.
Since $F|_{\KK(S)}$ is essentially surjective,
there
are idempotents $e$ and $f$ of $S$ such that
$e_0\mathrel{\mathscr D}F(e)$ and
$f_0\mathrel{\mathscr D}F(f)$.
Therefore, there are $x$ and $y$ in
the $\mathscr D$-class of $e_0$ such that
$e_0=xy$
and $F(e)=yx$,
and
there are $x'$ and $y'$ in
the $\mathscr D$-class of $f_0$ such that
$f_0=x'y'$
and $F(f)=y'x'$.
Let $t'=ytx'$.
Then $t'=ye_0 tf_0 x'=yxytx'y'x'=F(e)t'F(f)$,
thus $(F(e),t',F(f))$ is a morphism of $\KK(T)$.
Since $F$ is full,
there is $s\in S$ such that
$(F(e),t',F(f))=
F(e,s,f)$.

We have
$(e,s,f)=(e,ee_s,e_s)(e_s,s,f)$
and $(e_s,s,f)=(e_s,e_se,e)(e,s,f)$ and so $(e,s,f)\mathrel{\L} (e_s,s,f)$ in $\KK(S)$.
Also,
$(e_s,s,f)=(e_s,s,f_s)(f_s,f_sf,f)$
and
$(e_s,s,f_s)=(e_s,s,f)(f,ff_s,f_s)$,
thus $(e_s,s,f)\mathrel{\R}(e_s,s,f_s)$ in $\KK(S)$.  It follows that $(e,s,f)\mathrel{\mathscr D} (e_s,s,f_s)$ in $\KK(S)$.  Hence, using \eqref{olddefgood}, we have \[(F(e_s),F(s),F(f_s)) = F(e_s,s,f_s)\mathrel{\mathscr D} F(e,s,f) = (F(e),t',F(f))\] (in $\KK(T)$) and so $F(s)\mathrel{\mathscr D} t'$ by applying the semi-functor $\gamma_K\colon \KK(T)\to T$ that projects to the middle coordinate.
Since $xyt=t$
and $yt=t'y'$, we have
$t\mathrel{\L}yt\mathrel {\R}t'$.  Therefore,
$t$ and
$F(s)$ are $\mathscr D$-equivalent and hence isomorphic in $\DD(T)$.
This proves that $F$ is essentially surjective.

To prove that $F$ reflects regularity, we shall use freely that Lemma~\ref{l:isomorphism-criterion-for-C(S)} implies that an element of a semigroup $X$ is regular if and only if it is isomorphic in $\DD(X)$ to an element of $\KK(X)$.
Suppose that $s\in S$ is regular.   Then $s$ is isomorphic to some
object $e$ of $\KK(S)$ and hence $F(s)$ is isomorphic to $F(e)\in
\KK(T)$.  Thus $F(s)$ is regular.  Conversely, if $F(s)$ is regular,
then $F(s)$ is isomorphic to some object of $\KK(T)$ and hence, since the restriction $F\colon \KK(S)\to \KK(T)$ is an equivalence, to some object of $F(\KK(S))$.  Since equivalences reflect isomorphisms, it follows that $s$ is isomorphic to some object of $\KK(S)$ and so $s$ is regular.

To prove $F$ reflects the $\J$-order on $S$, let $s,s'\in S$.
By~\eqref{olddefgood} we have
\begin{align*}
F(e_s,s,f_s)&=(F(e_s),F(s),F(f_s))\\ F(e_{s'},s',f_{s'})&=(F(e_{s'}),F(s'),F(f_{s'})).
\end{align*}
Note that $F(e_s,s,f_s),F(e_{s'},s',f_{s'})\in \KK(T)$.  By Propositions~\ref{reflectsJ} and~\ref{Jorderinlocalunitscase} we deduce
\begin{align*}
s\leq_\J s'&\iff (e_s,s,f_s)\leq_\J (e_{s'},s',f_{s'}) \\ &\iff F(e_s,s,f_s)\leq_\J F(e_{s'},s',f_{s'})\iff F(s)\leq_\J F(s').
\end{align*}
This completes the proof.
 \end{proof}

Consider a functor $F\colon \KK(S)\to \KK(T)$, for a pair of semigroups $S$ and
$T$ with local units.  Our goal is to lift $F$ to a functor $\DD(S)\to \DD(T)$.  Moreover, if $F$ is an equivalence, the lifting will be a good functor.

  Let $(e_s)_{s\in S}$ and $(f_s)_{s\in S}$
  be families of idempotents of $S$ such that $s=e_ssf_s$,
  for each $s\in S$, and such that
  if $\varepsilon$ is an idempotent
  of $S$ then $e_\varepsilon=f_\varepsilon=\varepsilon$.  We will use without comment that $sS^1\cap S^1t\subseteq e_sSf_t$ for all $s,t\in S$.
  We define a functor $\widehat F\colon \DD(S)\to \DD(T)$
  sending an object $s$ of $\DD(S)$
  to $F_m(e_s,s,f_s)$ and
  a morphism
  $(s,u,t)$ to
  $(\widehat F(s),F_m(e_s,u,f_t),\widehat F(t))$.
  One must show that the latter triple is indeed an element of
  $\DD(T)$, a task which we include in the proof
  of Theorem~\ref{t:iso-SE-to-C(S)}.

\begin{Thm}\label{t:iso-SE-to-C(S)}
  Suppose that $S$ and $T$ are semigroups with local units
  and let $F\colon \KK(S)\to \KK(T)$ be a functor.
  Then $\widehat F\colon \DD(S)\to \DD(T)$, defined above, is a functor
  whose restriction to $\KK(S)$ coincides with $F$.
  Moreover, if $F$ is an equivalence, then $\widehat F$ is a good functor, and hence is an equivalence
  reflecting regularity and the $\J$-order on objects.
\end{Thm}

\begin{proof}
%
%
If $\varepsilon$ and $\phi$
are idempotents of $S$ then
\begin{equation*}
\widehat F(\varepsilon,u,\phi)
=(F_m(e_\varepsilon,\varepsilon,f_\varepsilon)
,F_m(e_\varepsilon,u,f_\phi),
F_m(e_\phi,\phi,f_\phi))
=F(\varepsilon,u,\phi),
\end{equation*}
  because the restrictions of the maps
  $s\mapsto e_s$ and $s\mapsto f_s$
  to $E(S)$ are the identities. Therefore, the restriction of $\widehat F$ to
  $\KK(S)$  coincides with $F$.

 We now show that $\widehat F(s,u,t)$  is indeed a morphism of $\DD(S)$.
 Let $x\in S^1$ be such that $u=sx$.
 Since $s=sf_s$ and $u=uf_t$, we may
 as well suppose
 that $x=f_sxf_t$.
Then
$(e_s,u,f_t)=(e_s,s,f_s)(f_s,x,f_t)$ in $\KK(S)$.
Since the restriction of $F_m$ to $\KK(S)$ is a semi-functor,
we conclude that $\widehat F(s)=F_m(e_s,s,f_s)\geq_\R F_m(e_s,u,f_t)$.
Similarly, $\widehat F(t)\geq_\L F_m(e_s,u,f_t)$.  It follows that $(\widehat F(s),F_m(e_s,u,f_t),\widehat F(t))$ is a morphism of $\DD(T)$.

Next, we prove that $\widehat F$ is a functor.
Clearly, $\widehat F(1_s)=(\widehat F(s),\widehat F(s),\widehat F(s))=1_{\widehat F(s)}$ by construction.
Let $(s,u,t)$ and $(t,v,r)$ be two composable arrows of
$\DD(S)$. Let $x$ be such that $u=xt$.
We may suppose that $x=e_sxe_t$.
Then
\begin{equation}\label{eq:hat-theta-is-functor-part1}
  \widehat F((s,u,t)(t,v,r))=
  \widehat F(s,xv,r)=(\widehat F(s),F_m(e_s,xv,f_r),\widehat F(r)).
\end{equation}
On the other hand, from $(e_s,u,f_t)=(e_s,x,e_t)(e_t,t,f_t)$, we have  $F_m(e_s,u,f_t)
  = F_m(e_s,x,e_t)\cdot\widehat F(t)$
and so
\begin{align}
  \widehat F(s,u,t)\cdot\widehat F(t,v,r)
  &=(\widehat F(s),F_m(e_s,u,f_t),\widehat F(t))\cdot
  (\widehat F(t),F_m(e_t,v,f_r),\widehat F(r)).\notag\\
&=(\widehat F(s),F_m(e_s,x,e_t)F_m(e_t,v,f_r),\widehat F(r)).\label{eq:hat-theta-is-functor-part2}
\end{align}
Comparing \eqref{eq:hat-theta-is-functor-part1}
and \eqref{eq:hat-theta-is-functor-part2}, we obtain
  $\widehat F((s,u,t)(t,v,r))=\widehat F(s,u,t)\cdot\widehat F(t,v,r)$.
  Thus $\widehat F$ is a functor.

  It remains to prove that if $F$ is an equivalence, then $\widehat F$ is a good functor.  But this is immediate from the definition since \[\widehat F(e_s,s,s) = (\widehat F(e_s),F_m(e_s,s,f_s),\widehat F(s)) = (\widehat F(e_s),\widehat F(s),\widehat F(s))\] using that $e_{e_s}=e_s$. Similarly $\widehat F(s,s,f_s) = (\widehat F(s),\widehat F(s),\widehat F(f_s))$.  Thus $F$ is good, completing the proof in light of Lemma~\ref{goodequivalence}.
\end{proof}

\begin{Cor}\label{c:implication-k-to-d}
  Let $S$ and $T$ be semigroups with local units.
  If $\KK(S)$ and $\KK(T)$ are equivalent, then so are $\DD(S)$ and $\DD(T)$.
\end{Cor}

We do not know if the converse of Corollary~\ref{c:implication-k-to-d} holds.

\section{Actions of $\KK(S)$ and $\DD(S)$}\label{sec:actions}


From here on out, we mention only right actions, and so \emph{action} will be
synonymous of \emph{right action} until the end of this paper.

An \emph{action} of a small category $C$ on a set $Q$ is a contravariant
functor $\mathbb A\colon C\to \pv{Set}$
such that $\mathbb A(c)$ is a subset of $Q$ for every object $c$ of $C$, i.e., a \emph{presheaf} on $C$ taking values in subsets of $Q$~\cite{MM-Sheaves}.
If $s\colon c\to d$ is a morphism of $C$,
we may use the notation
$q\cdot s$ for $\mathbb A(s)(q)$, where $q\in \mathbb A(d)$,
and the notation $\mathbb A(d)\cdot s$
for the image of the function $\mathbb A(s)\colon \mathbb A(d)\to
\mathbb A(c)$. The notation $q\cdot s$ and $\mathbb A(d)\cdot s$
may be simplified to $qs$ and $\mathbb A(d)s$.

\begin{Def}[Equivalent actions]\label{def:equivalent-actions}
Consider two actions $\mathbb A\colon C\to \pv{Set}$
and $\mathbb A'\colon D\to \pv{Set}$.
We say that $\mathbb A$ and $\mathbb A'$ are
\emph{equivalent}, and write $\mathbb A\sim\mathbb A'$, if
there is an equivalence $F\colon C\to D$
such that $\mathbb A$ and $\mathbb A'\circ F$ are isomorphic functors.
\end{Def}

\begin{Rmk}\label{r:equivalence-of-actions}
Note that the binary relation $\sim$ is an equivalence relation on
the class of actions. It is clearly reflexive;
transitivity and symmetry follow straightforwardly from the
fact that, for small categories $C$ and $D$, if $F\colon C\to D$ and
 $G\colon C\to D$ are isomorphic functors, then
  $FH$ and $GH$ are also isomorphic for every functor $H$
 from a category into $C$.
\end{Rmk}

\begin{Def}\label{def:action-of-K(S)}
  Consider an action of a semigroup $S$ on a set $Q$.
  Let $\mathbb A_Q$ be the action of $\KK(S)$ on $Q$ such that
  $\mathbb A_Q(e)=Qe$ for
  every object $e$ of $\KK(S)$,
  and such that $\mathbb A_Q((e,s,f))$
  is the function
$Qe\to Qf$
mapping $q$ to $qs$, for every $q\in Qe$;
that is, $q\cdot (e,s,f)=q\cdot s$.

Similarly, let $\mathbb B_Q$ be the
action of $\DD(S)$ on $Q$
such that $\mathbb B_Q(s)=Qs$ for each object
$s$ of $\DD(S)$, and where $\mathbb B_Q((s,u,t))$
is the function
$Qs\to Qt$
mapping $qs$ to $qu$, when $q\in Q$.
That is, $qs\cdot (s,u,t)=qu$.
Note that $\mathbb B_Q((s,u,t))$ is well defined:
indeed, for all $q_1,q_2\in Q$, if
$q_1s=q_2s$, writing $u=sx$ with $x\in S^1$, we have $q_1u=q_1sx=q_2sx=q_2u$.
Note also that the restriction of $\mathbb B_Q$
to $\KK(S)$ is~$\mathbb A_Q$.
\end{Def}


\begin{Lemma}\label{l:S-faithful-implies-KS-also}
  If the action of the semigroup $S$ on $Q$ is faithful, then
  the actions $\mathbb A_Q$ and $\mathbb B_Q$ are faithful.
\end{Lemma}

\begin{proof}
 Let $(s,u,t)$, $(s,v,t)$ be distinct coterminal morphisms of $\DD(S)$.
 By faithfulness for $S$ we can find $q\in Q$ with $qu\neq qv$.  Then,
 by the definition of $\mathbb B_Q$,
 we have $qs(s,u,t)=qu$ and $qs(s,v,t)=qv$.
 We conclude $\mathbb B_Q$ is faithful.
  As $\mathbb A_Q$ is the restriction of $\mathbb B_Q$ to $\KK(S)$,
  the action $\mathbb A_Q$ is also faithful.
\end{proof}

The remaining part of this section
concerns the lifting of an equivalence
$\mathbb A_Q\sim \mathbb A_R$ to an equivalence
$\mathbb B_Q\sim \mathbb B_R$.
Theorem~\ref{t:iso-SE-to-C(S)} will be crucial. The following proposition
deals with other technicalities that need to be addressed.

\begin{Prop}\label{p:lifting-eta-to-lambda}
  Let $S$ and $T$ be semigroups with local units.
  Suppose that $S$ acts on the set $Q$ and 
  $T$ acts on the set $R$.
  Let $F\colon \KK(S)\to \KK(T)$ be a functor
  and  $\eta\colon \mathbb A_Q\Rightarrow \mathbb A_R\circ F$ a natural transformation. Retaining the notation of Theorem~\ref{t:iso-SE-to-C(S)},
  if $s\in S$ and $q\in Q$,
  then~$\eta_{f_s}(qs) =\eta_{e_s}(qe_s)\cdot\widehat F(s)$.  Therefore,
  the map $\eta_{f_s}\colon Q\cdot f_s\to R\cdot F(f_s)$
    restricts to a  mapping $\lambda_s\colon Q\cdot s\to R\cdot \widehat F(s)$.
  Moreover, the family $\lambda=(\lambda_s)_{s\in S}$ constitutes
    a natural transformation
    $\lambda\colon \mathbb B_Q\Rightarrow \mathbb B_R\circ \widehat F$.
    Furthermore, if $\eta$ is a natural isomorphism, then so is $\lambda$.
\end{Prop}

\begin{proof}
    Note that, for all $s\in S$, we have
    $Q\cdot s\subseteq Q\cdot f_s$.
   Moreover, since $\eta$ is a natural transformation,
   the following chain of equalities holds for each~$q\in Q$:
  \begin{equation}\label{eq:we-can-restrict}
  \eta_{f_s}(qs)=\eta_{f_s}(qe_s\cdot (e_s,s,f_s))=
  \eta_{e_s}(qe_s)\cdot F(e_s,s,f_s)
  =\eta_{e_s}(qe_s)\cdot\widehat F(s).
  \end{equation}
  Hence, we may
  restrict the map $\eta_{f_s}$ to a map
  $\lambda_s\colon
  Q\cdot s\to R\cdot \widehat F(s)$.
  We wish to show that the family $\lambda=(\lambda_s)_{s\in S}$ yields a
  natural transformation $\lambda\colon \mathbb B_Q\Rightarrow \mathbb B_R\circ \widehat F$.
  That is, we want to prove that the following diagram is commutative
  \begin{equation}\label{eq:ds-commutes}
    \begin{split}
        \xymatrix{
      Q\cdot s\ar[rr]^-{q\mapsto q\cdot(s,u,t)}\ar[d]_{\lambda_s}
      &&Q\cdot t\ar[d]^{\lambda_t}
      \\
      R\cdot \widehat F(s)\ar[rr]_-{r\mapsto r\cdot \widehat F(s,u,t)}
      &&R\cdot \widehat F(t).
    }
    \end{split}
  \end{equation}
  for each arrow $(s,u,t)$ of $\DD(S)$.
  Let $q$ be an element of $Q$. Then
  $\lambda_s(qs)=\eta_{e_s}(qe_s) \widehat F(s)$ by~\eqref{eq:we-can-restrict}.
  We then have the following chain of equalities,
  equality~\eqref{eq:chain-2} holding
  by the definition of $\widehat F$
  and equality~\eqref{eq:chain-4} because $\eta$ is a natural transformation:
  \begin{align}
    \lambda_s(qs)\cdot \widehat F(s,u,t)
    &=\eta_{e_s}(qe_s) \widehat F(s)\cdot \widehat F(s,u,t)\notag\\
    &=\eta_{e_s}(qe_s)\widehat F(s)\cdot (\widehat F(s),F_m(e_s,u,e_t),\widehat F(t))\label{eq:chain-2}\\
    &=\eta_{e_s}(qe_s)\cdot F_m(e_s,u,f_t)\notag\\
    &=\eta_{e_s}(qe_s)\cdot F(e_s,u,f_t)\notag\\
    &=\eta_{f_t}\bigl(qe_s\cdot (e_s,u,f_t)\bigr)\label{eq:chain-4}\\
    &=\eta_{f_t}\bigl(qu).\notag
  \end{align}
  Since $qu\in Qt$, we have $\eta_{f_t}\bigl(qu)=\lambda_t\bigl(qu)$,
  thus we have shown
  \begin{equation*}
    \lambda_s(qs)\cdot \widehat F(s,u,t)=\lambda_t\bigl(qu).
  \end{equation*}
By the definition of the action of $\DD(S)$ on
$Q$, we also have $qs\cdot(s,u,t)=qu$. Therefore
$\lambda_s(qs)\cdot \widehat F(s,u,t)=\lambda_t(qs\cdot(s,u,t) )$,
which establishes the commutativity of~(\ref{eq:ds-commutes}).

Finally, suppose that $\eta$ is a natural isomorphism.
For every object~$s$ of $\DD(S)$,
the function $\lambda_s$ is  obviously injective, since it is a
restriction of the injective function $\eta_{e_s}$.
It remains to prove that $\lambda_s$ is onto.
Let $r\in R$. Then we have
\begin{equation}\label{eq:onto-lamda-s}
  r\widehat F(s)=rF_m(e_s,s,f_s)=rF(e_s)\cdot F(e_s,s,f_s).
\end{equation}
Since $\eta_{e_s}$ is onto, there is $q\in Qe_s$
such that $r F(e_s)=\eta_{e_s}(q)$.
As $\eta$ is a natural transformation,
it then follows from~\eqref{eq:onto-lamda-s} that
$r\widehat F(s)=\eta_{f_s}(q\cdot (e_s,s,f_s))=\lambda_s(qs)$.
Therefore $\lambda_s$ is onto.
\end{proof}


\begin{Cor}\label{c:lifting-eta-to-lambda}
    Let $S$ and $T$ be semigroups with local units with actions on $Q$ and $R$, respectively. If $\mathbb A_Q$ and $\mathbb A_R$ are equivalent actions,
  then there is an equivalence $G\colon \DD(S)\to \DD(T)$
  and a natural isomorphism
  $\lambda\colon \mathbb B_Q\Rightarrow \mathbb B_R\circ G$
   such that
   \begin{enumerate}
   \item $G$ restricts to an equivalence
     $\mathbb \KK(S)\to \mathbb \KK(T)$;
   \item $G$ reflects regularity and the $\J$-order on
   objects;
   \item $\lambda$ restricts to a
   natural isomorphism
   $\mathbb A_Q\Rightarrow\mathbb A_R\circ G|_{\mathbb \KK(S)}$.
   \end{enumerate}
\end{Cor}

\begin{proof}
  This is an immediate consequence
of Theorem~\ref{t:iso-SE-to-C(S)}
and Proposition~\ref{p:lifting-eta-to-lambda}: just take $G=\widehat F$,
where $F$ is an equivalence $\KK(S)\to \KK(T)$
such that $\mathbb A_Q$ and $\mathbb A_R\circ F$ are isomorphic.
\end{proof}

In particular, we have the following.

\begin{Cor}
  Let $S$ and $T$ be semigroups with local units with actions on $Q$ and $R$, respectively.
  If $\mathbb A_Q\sim \mathbb A_R$ then $\mathbb B_Q\sim \mathbb B_R$.\qed
\end{Cor}

\section{Karoubi invariants of an action}
Let $C$ be a category.  An assignment of an object $F(Q,S)$ of $C$ to each action of a semigroup $S$  on a set $Q$ is said to be a \emph{Karoubi invariant} of the action if $F(Q,S)$ is isomorphic to $F(Q',T)$ whenever the actions $\mathbb A_Q$ and $\mathbb A_{Q'}$ are equivalent. This section includes several Karoubi invariants used in~\cite{categoricalinvariant} to obtain flow equivalence invariants for symbolic dynamical systems.

\subsection{The poset of an action}
Suppose that a semigroup $S$ acts on a set $Q$.  Let $I=Q\cdot E(S)$ and define a preorder $\preceq$ on $I$ by $q\preceq q'$ if $q'LU(S)\subseteq qLU(S)$, that is, $q'=qs$ for some $s\in LU(S)$.  As usual, define $q\sim q'$ if $q\preceq q'$ and $q'\preceq q$.  Then $P(Q)=I/{\sim}$ is a poset, which can be identified with the poset of cyclic $LU(S)$-subsets $\{q\cdot LU(S)\mid q\in I\}$) ordered by reverse inclusion. We show that $P(Q)$ is a Karoubi invariant of the action of $S$ on $Q$.  More precisely, we have the following result.

\begin{Thm}\label{t:actionposet}
Suppose $S,T$ are semigroups acting on sets $Q,Q'$, respectively.  Suppose that $F\colon \KK(S)\to \KK(T)$ is an equivalence of categories and that $\eta\colon \mathbb A_Q\Rightarrow \mathbb A_{Q'}\circ F$ is an isomorphism.  Then there is a well-defined isomorphism $f\colon P(Q)\to P(Q')$ of posets given by $f(qLU(S)) = \eta_e(q)LU(T)$ for $q\in Qe$. In particular, $P(Q)$ is a Karoubi invariant.
\end{Thm}
\begin{proof}
We claim that if $q\in Qe$ and $q'\in Qe'$ with $e,e'\in E(S)$, then $q'\in qLU(S)$ if and only if there exists an arrow $(e,s,e')$ of $\KK(S)$ with $q'=q(e,s,e')$.  Indeed, if $q'=qs$ with $s\in LU(S)$, then $q'=q'e'=qse'=qese'=q(e,ese',e')$.  Conversely, if $q'=q(e,s,e')$, then $q'=qs$ and $s=ese'\in LU(S)$ and so $q'\in qLU(s)$. A similar claim, of course, holds for $Q'$ and $T$.

To see that $f$ is a well-defined order preserving map, let $q,q'\in QE(S)$ and suppose that $qLU(S)\supseteq q'LU(S)$. Choose $e,e'\in E(S)$ with $qe=q$ and $q'e'=q'$.  Then by the claim, there is an arrow $(e,s,e')$ with $q(e,s,e')=q'$.  Therefore, we have  $\eta_{e'}(q') = \eta_{e'}(q(e,s,e')) = \eta_e(q)F(e,s,e')$ and so $\eta_e(q)LU(T)\supseteq \eta_{e'}(q')LU(T)$ by the claim.  Thus $f$ is well defined and order preserving.

To show that $f$ is an order embedding, suppose that $qe=q, q'e'=q'$ (with $e,e'\in E(S)$) and $\eta_{e}(q)LU(T)\supseteq \eta_{e'}(q')LU(T)$.  Then by the claim there is an arrow $(F(e),t,F(e'))$ of $\KK(T)$ with $\eta_e(q)(F(e),t,F(e')) = \eta_{e'}(q')$.  Since $F$ is full, there is an arrow $(e,s,e')$ of $\KK(S)$ with $F(e,s,e')=(F(e),t,F(e'))$.  Then
\[\eta_{e'}(q(e,s,e')) =\eta_e(q)(F(e),t,F(e'))=\eta_{e'}(q')\] and hence $q'=q(e,s,e')$ because $\eta_{e'}\colon Qe'\to Q'F(e')$ is a bijection.  Thus $qLU(S)\supseteq q'LU(S)$ by the claim.  We conclude that $f$ is an order embedding.

It remains to show that $f$ is surjective.  Let $q'\in Q'E(T)$.  Choose $e'\in E(T)$ with $q'e'=q'$.  Since $F$ is essentially surjective, there is an isomorphism $(e',t,F(e))$ with $e\in E(S)$; say its inverse is $(F(e),t',e')$.  Let $q=\eta_{e}\inv(q'(e',t,F(e)))$.  Then $\eta_e(q)=q'(e',t,F(e))$ implies that $q'LU(T)\supseteq \eta_e(q)LU(T)$ by the claim.  Conversely, we compute  $\eta_e(q)(F(e),t',e') = q'(e',t,F(e))(F(e),t',e') = q'$ and so $\eta_e(q)LU(T)\supseteq q'LU(T)$ by the claim.  Thus $f(qLU(S))=\eta_e(q)LU(T)=q'LU(T)$ and hence $f$ is surjective.  This proves the theorem.
\end{proof}


\subsection{Labeled preordered sets of the $\mathscr D$-classes}\label{sec:label-preord-set}

\subsubsection{Regularity and the  Sch\"utzenberger group of a
  $\mathscr D$-class are preserved by equivalences}

Given a semigroup $S$, let $\mathfrak {D}(S)$ be the set
of $\mathscr D$-classes of $S$.
Endow $\mathfrak {D}(S)$ with the preorder
$\preceq$ such that, if $D_1$ and $D_2$ belong to $\mathfrak {D}(S)$,
then $D_1\preceq D_2$ if and only if
there are $d_1\in D_1$ and $d_2\in D_2$ such that $d_1\leq_{\J}d_2$.
Note that if $\ci D=\ci J$ (for example, if $S$ is finite), then
$\preceq$ is a partial order.

If we assign to each element $x$ of a preordered set $P$
a label $\lambda_P(x)$ from some set, we obtain a new structure,
called \emph{labeled preordered set}.
A morphism in the category of
labeled preordered sets
is a morphism $\varphi\colon P\to Q$ of preordered sets
such $\lambda_Q\circ\varphi=\lambda_P$.

For a semigroup $S$, assign to each element $D$ of $\mathfrak {D}(S)$
the label $\lambda(D)=(\varepsilon,H)$
where $\varepsilon\in\{0,1\}$, with $\varepsilon=1$ if and only if $D$
is regular, and $H$ is the
Sch\"utzenberger group of $D$.
We denote the labeled preordered set thus obtained by
$\mathfrak {D}_\ell(S)$.

\begin{Rmk}\label{r:the-same-relations}
   Let $S$ be a semigroup, not necessarily with local units.
   If $\mathscr K$ is one of Green's relations $\mathscr J$, $\mathscr R$ or $\mathscr L$,
   then $s\leq_{\mathscr K}t$ in ${LU}(S)$ if and only if
 $s\leq_{\mathscr K}t$ in $S$, for all $s,t\in {LU}(S)$,
 and the Sch\"utzenberger group of $s\in{LU}(S)$ is the same in $S$
 as in ${LU}(S)$ (cf.~proof of~\cite[Proposition 5.2]{Costa:2006}).
 Therefore, $\mathfrak {D}_\ell\bigl({LU}(S)\bigr)$ is
 obtained from
 $\mathfrak {D}_\ell\bigl(S)$ by removing the
 $\mathscr D$-classes of $S$
 not contained in $LU(S)$.
 \end{Rmk}


\begin{Thm}\label{t:isomorphism-of-preordered-sets-general}
  Let $S$ and $T$ be semigroups.
  If $\KK(S)$ and $\KK(T)$ are equivalent categories, then
  $\mathfrak {D}_\ell(LU(S))$ and $\mathfrak {D}_\ell(LU(T))$
  are isomorphic labeled preordered sets.
\end{Thm}

Before proving
Theorem~\ref{t:isomorphism-of-preordered-sets-general},
we mention that, with a different language,
its content is almost entirely proved
in~\cite[Proposition~5.1]{Lawson:2011},
the exception being the part concerning the labeling by
Sch\"utzenberger groups.
Our proof fills this gap easily because of the use of
the Sch\"utzenberger category.

  Let $S$ and $T$ be semigroups with local units.
Let $G\colon \DD(S)\to \DD(T)$ be an equivalence.
Denote by $\bar G\colon \mathfrak {D}(S)\to \mathfrak {D}(T)$
the function such that $\bar G(D_s)=D_{G(s)}$, where $D_t$ denotes the
$\mathscr D$-class of $t$.
    Note that $\bar G$ is
    well defined, because $G$ preserves isomorphisms.

\begin{Lemma}\label{l:isomorphism-of-preordered-sets}
  Suppose that the equivalence
  $G\colon \DD(S)\to \DD(T)$
  reflects regularity and the $\J$-order on objects.
 Then the map $\bar G$ is an isomorphism
 from $\mathfrak {D}_\ell(S)$ to $\mathfrak {D}_\ell(T)$.
\end{Lemma}

\begin{proof}
    Let $H$ be a quasi-inverse of $G$. Then
    $\bar H$ and $\bar G$ are functional inverses of each other.
    Moreover, $\bar G$
    and $\bar H$
    are morphisms of preordered sets
    because $G$ reflects the $\J$-order.
  Since equivalences preserve the automorphism groups of
  objects,
  from Lemma~\ref{l:automorphism-group}
  and the hypothesis that $G$ reflects regularity,
  we conclude that $\bar G$ preserves labels.
\end{proof}

\begin{proof}[Proof of Theorem~\ref{t:isomorphism-of-preordered-sets-general}]
Without loss of generality, we may assume that $S$ and $T$ have local units.  Theorem~\ref{t:iso-SE-to-C(S)} guarantees the existence of an
equivalence $\DD(S)\to \DD(T)$ that
reflects regularity and the $\J$-order on objects, whence
Theorem~\ref{t:isomorphism-of-preordered-sets-general}
follows immediately from Lemma~\ref{l:isomorphism-of-preordered-sets}.
\end{proof}

\subsubsection{The rank of a $\mathscr D$-class is preserved under equivalences}

Suppose we have a (right) action of a semigroup $S$ on a set $Q$.
Then each element of $S$ can be viewed as a transformation
on $Q$.
Recall that the \emph{rank} of a transformation is the cardinality of its
image, and that $\mathscr D$-equivalent elements of $S$ have
the same rank as transformations of $Q$.
We modify the labeled preordered set
$\mathfrak {D}_\ell\bigl(S)$ as follows: for each
$\mathscr D$-class $D$ of $S$, instead of the label $(\varepsilon,H)$, consider
the label $(\varepsilon,H,r)$, where $r$ is the rank
in $Q$ of an element of $D$, viewed as an
element of the transformation semigroup of $Q$
defined by the action of $S$ on $Q$.
Denote by $\mathfrak {D}_Q\bigl(S)$ the resulting labeled poset.

\begin{Thm}\label{t:preordered-sets-with-ranks}
  Let $S$ and $T$ be semigroups.
  Suppose that $S$ acts on the set $Q$, and $T$ acts on the set $R$.
  If $\mathbb A_Q$ and $\mathbb A_R$ are equivalent actions, then
  $\mathfrak {D}_Q(LU(S))$ and $\mathfrak {D}_R(LU(T))$
  are isomorphic labeled preordered sets. Thus $Q\mapsto \mathfrak{D}_Q(LU(S))$ is a Karoubi invariant.
\end{Thm}

\begin{proof}
Replacing $S$ by $LU(S)$ and $T$ by $LU(T)$, we may assume without loss of generality that $S$ and $T$ have local units.
  By Corollary~\ref{c:lifting-eta-to-lambda},
  there is an equivalence
  $G\colon \DD(S)\to \DD(T)$
  such that $G$ reflects
  regularity and the $\J$-order on objects,
  and such that $\mathbb B_Q$ and
  $\mathbb B_R\circ G$ are equivalent functors.
  By Lemma~\ref{l:isomorphism-of-preordered-sets},
  the map
  $\bar G\colon\mathfrak {D}_\ell(S)
  \to\mathfrak {D}_\ell(T)$
  is an isomorphism
  of labeled preordered sets. Therefore,
  to prove that
  $\mathfrak {D}_Q(S)$ and $\mathfrak {D}_R(T)$
  are isomorphic labeled preordered sets, it remains to show
  that $s$ and $G(s)$
  have the same rank,
  for every $s\in S$.
  But the rank of $s$ in $Q$ is the cardinality of
  $\mathbb B_Q(s)$, and the rank of $G(s)$
  in $R$ is the cardinality of $\mathbb B_R(G(s))$.
  Therefore, what we want to show follows from $\mathbb B_Q$ and
  $\mathbb B_R\circ G$ being equivalent functors.
\end{proof}





\bibliographystyle{abbrv}


\providecommand{\bysame}{\leavevmode\hbox to3em{\hrulefill}\thinspace}
\providecommand{\MR}{\relax\ifhmode\unskip\space\fi MR }
\providecommand{\MRhref}[2]{%
  \href{http://www.ams.org/mathscinet-getitem?mr=#1}{#2}
}
\providecommand{\href}[2]{#2}

\end{document}